\documentclass[11pt]{article}

\usepackage[a4paper,margin=1in]{geometry}
\usepackage[T1]{fontenc}
\usepackage[utf8]{inputenc}
\usepackage{lmodern}
\usepackage{amsmath,amsthm,amssymb,mathtools}
\usepackage{booktabs,multirow,makecell,array}
\usepackage{authblk}
\usepackage{enumitem}
\usepackage{xcolor}
\usepackage{listings}
\usepackage{graphicx}
\usepackage[
  backend=biber,
  style=numeric,
  sorting=none,
  natbib=true,
  doi=true,
  url=true,
  isbn=false,
  eprint=false
]{biblatex}
\usepackage{hyperref}

\hypersetup{
  colorlinks=true,
  linkcolor=blue,
  citecolor=blue,
  urlcolor=blue,
  breaklinks=true
}

\newtheorem{theorem}{Theorem}[section]
\newtheorem{proposition}[theorem]{Proposition}
\newtheorem{lemma}[theorem]{Lemma}
\newtheorem{corollary}[theorem]{Corollary}
\newtheorem{remark}[theorem]{Remark}

\newcommand{\Cxi}{C^{\xi}}

\title{A SAT-based Filtering Framework for Exact Coverings of $K_{33}$ by Cliques of Order 3, 4 or 5}
\author[1]{Petr Kova\v{r}}
\author[1,2,3]{Yifan Zhang}
\affil[1]{Department of Applied Mathematics, VSB -- Technical University of Ostrava}
\affil[2]{Department of Mathematics, University of Ostrava}
\affil[3]{Department of Algebra, Charles University}
\date{}

\begin{document}
\maketitle

\begin{abstract}
	We investigate the minimum number of cliques of orders $3$, $4$, and $5$ needed to cover the edges of $K_{33}$ with zero excess. General covering results yield the lower bound 57. The main result of the paper is that no decomposition of $K_{33}$ into $57$ blocks from $\{K_3,K_4,K_5\}$ exists.
	
	Our approach is algorithmic and relies on a layered exact-search pipeline rather than a single monolithic solver. We combine symmetry reduction, enumeration of local signatures, arithmetic profile restrictions, geometric tests for partial configurations, SAT realisation on reduced instances, and final decoding checks. The benchmark comparison shows that this structured approach is substantially more effective than direct ILP, DLX, or SAT formulations on the full problem.
	
	As a consequence, we obtain $\Cxi(33,\{3,4,5\},2)\ge 58$. A short additional counting argument further strengthens this to $\Cxi(33,\{3,4,5\},2)\ge 59$. We also give new compressed proofs for the known exceptional cases $K_{18}$ and $K_{19}$ in the setting of $\{K_3,K_4\}$-decompositions, illustrating the same combination of theoretical reduction and exact computation.
	
	Finally, we explain the relevance of the $K_{33}$ result to the open packing problem of determining the packing number $D(33,5,2)$. A packing of $51$ copies of $K_5$ in $K_{33}$ would leave a $4$-regular graph on $9$ vertices, and our exclusion already rules out two natural candidate leave structures.
\end{abstract}

\noindent\textbf{Keywords:} exact covering, clique decomposition, complete graph, SAT-assisted search, symmetry reduction

\noindent\textbf{MSC Classification: Primary 05C70; Secondary 05B40 , 05C85.}

\section{Introduction}

Decompositions and coverings of complete graphs by small cliques lie at the intersection of design theory, extremal graph theory, and exact combinatorial search. In the mixed-size setting one allows more than one block size, which typically makes the arithmetic as well as the structure harder: a single edge count no longer determines the block composition, and multiple local configurations may survive the first congruence tests.

In this paper we consider coverings of complete graphs by cliques of orders $3$, $4$, and $5$. For a covering, edges covered more than once form the \emph{excess}. We first minimise the excess and then, among coverings with minimum excess, minimise the number of covering cliques. For fixed allowed block sizes $\{3,4,5\}$, we write $\Cxi(v,\{3,4,5\},2)$ for the minimum number of blocks in a covering of $K_v$ with minimum excess. A zero-excess covering is called \emph{exact}, and it corresponds in fact to a graph decomposition. We always denote by $\alpha,\beta,\gamma$ the number of triples, quadruples and quintuples in a covering design, respectively. Their local version $\alpha_x,\beta_x,\gamma_x$ denotes the number of triples, quadruples and quintuples containing a given vertex $x$, respectively. 

For $v=33$, the general theory already yields two important facts: the minimum excess is zero, and the general lower bound is $57$ \cite{k3k4k5paper}. The open question is therefore how far one can push the zero-excess lower bound beyond $57$. The main objective of the present paper is to explain a successful algorithmic strategy for doing exactly that.

The core message is not simply that one particular case is impossible. Rather, it is that large mixed-clique covering instances often resist direct attacks by a single solver. A monolithic formulation using Integer Linear Programming (ILP), Satisfiability Test (SAT), or Dancing Links Algorithm X (DLX) can be written down immediately, but in hard infeasible cases (or UNSAT using the language of satisfiability) this often leads to very long runs with little structural information returned to the user. Our approach is to replace one global exact-search problem by a sequence of smaller exact decisions. At each layer, some branches are excluded by symmetry, some by arithmetic, some by local geometry, some by SAT solver, and some only after decoding SAT models back into the combinatorial language and checking edge-disjointness or other constraints again. In practice, this avoids the ``eternal waiting'' typical of large unsatisfiable monolithic runs.

A secondary aim of the paper is expository. We briefly review the three most common exact-search paradigms for clique coverings---ILP, DLX, and SAT---and explain how our layered method is built on top of them rather than against them. We also include a compressed warm-up proof for the minimum exact $\{K_3,K_4\}$-cover of $K_{18}$ and $K_{19}$, because these cases already exhibit the same principle: use theory first to isolate a handful of sharply structured residues, and call the solver only on those residues. 

\paragraph*{Main results.}
The principal algorithmic theorem of the paper is the following.

\begin{theorem}\label{thm:main-57}
There is no decomposition of $K_{33}$ into $57$ cliques of orders from $\{3, 4, 5\}$.
\end{theorem}

This immediately implies $\Cxi(33,\{3,4,5\},2)\ge 58$. In fact, the argument can be pushed one step further.

\begin{theorem}\label{thm:main-59}
We have
$\Cxi(33,\{3,4,5\},2)\ge 59$.
\end{theorem}

The proof of Theorem~\ref{thm:main-59} is short once Theorem~\ref{thm:main-57} is established, so we include it as a final section. Nevertheless, the main methodological contribution of the paper remains the algorithmic exclusion of the $57$-block case.

\section{Three standard exact paradigms}

This section briefly recalls the three exact paradigms that underlie most computational work on covering problems: integer linear programming, dancing links, and SAT.

\subsection{Integer linear programming (ILP)}

Let $R$ be the family of candidate cliques of allowed sizes. For each $S\in R$ we introduce a binary decision variable $x_S$, with $x_S=1$ meaning that the clique on vertex set $S$ is selected. If $\mu(e)$ denotes the required covering multiplicity of an edge $e$, then the exact-multiplicity model is
\[
\sum_{S\in R\, :\, e\subseteq S} x_S = \mu(e) \qquad (e\in E(K_v)).
\]
For an exact covering without excess, i.e. a decomposition, we have $\mu(e)=1$ for every edge $e$. If we additionally fix the numbers of blocks of each size, we impose
\[
\sum_{|S|=3} x_S = \alpha, \qquad \sum_{|S|=4} x_S = \beta, \qquad \sum_{|S|=5} x_S = \gamma.
\]
This is especially convenient when earlier combinatorial arguments have already reduced the problem to a small list of admissible triples $(\alpha,\beta,\gamma)$. Standard background on ILP and computational complexity may be found in \cite{Cook1971,GareyJohnson1979}, while the implementation used in our computations is based on Gurobi \cite{gurobi}.

ILP is attractive because it can produce both a feasible configuration and a rigorous lower bound from the relaxation. In small and medium instances it is often the most convenient first tool. In the $K_{18}$ and $K_{19}$ warm-up cases, for example, ILP is used mainly as a final completion check after structural pruning.

\subsection{Dancing links (DLX)}

Knuth's dancing links algorithm treats exact cover as a sparse incidence problem \cite{knuth_2000}. In clique decompositions, the columns correspond to edges and the rows to candidate blocks. The original exact-cover setting fits decompositions directly; for coverings with prescribed multiplicities one uses a multiplicity-aware generalisation of the algorithm, which has already been applied to clique coverings of complete graphs \cite{Kovar2023}.

The main advantage of dancing links is that the data structure supports very fast reversible branching. For pure decomposition problems or small fixed-excess instances, it can outperform more algebraic approaches. Its main weakness here is the same as for any monolithic exhaustive search: in large UNSAT instances, even a well-implemented exact-cover search may spend a long time exploring a huge symmetric space before certifying nonexistence.

On a computer with a 12th Gen Intel(R) Core(TM) i7-12700H CPU and 20 cores, we test both the ILP and dancing links with the simple problem of decomposing a complete graph of order $n$ into triangles, where $n\equiv 1,3\pmod{6}$. See Figure \ref{fig:knbyk3upto100} for the comparison between the dancing links algorithm and the integer linear programming for $n<100$. Evidently, the dancing links algorithm finds an existing solution more efficiently than the integral programming. 

\begin{figure}
	\centering
	\includegraphics[scale=0.56]{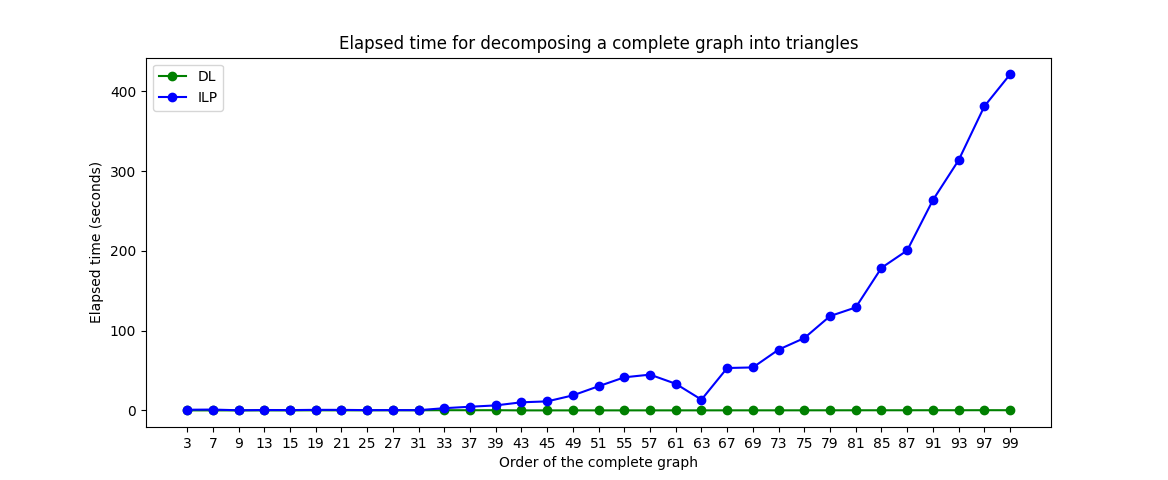}
	\caption{Comparison of the dancing links (1 solution) with integer linear programming}
	\label{fig:knbyk3upto100}
\end{figure}

\subsection{SAT}

A SAT encoding uses one Boolean variable $x_S$ for each candidate clique and expresses exact edge multiplicities and block counts by CNF clauses. There has been some attempt to apply SAT in graph decomposition problems using clever symmetry-breaking methods \cite{Zhao_2017}. Modern CDCL solvers are very effective once the parameters are fixed and the instance has already been substantially reduced. In particular, SAT is a natural feasibility engine for ``small but still exact'' residual problems; see \cite{Biere2009} for general background. The solver used in our computations is \texttt{kissat} \cite{BiereFazekasFleuryHeisinger2020,KissatWeb}.

\paragraph*{Why not use one monolithic encoding?}

For the full $K_{33}$ problem with allowed block sizes $3$, $4$, and $5$, the raw candidate family already has size
\[
\binom{33}{3}+\binom{33}{4}+\binom{33}{5}=5~456+40~920+237~336=283~712.
\]
Thus a monolithic exact encoding begins with $283{,}712$ candidate rows or variables before any symmetry is removed. The number of edges is only $\binom{33}{2}=528$, so the difficulty is not the number of basic constraints, but the enormous branching space together with heavy symmetry and the fact that the target instances are unsatisfiable.

Even after the arithmetic reduction relevant for the $57$-block lower bound, namely to the unique composition $(\alpha,\beta,\gamma)=(6,0,51)$, a direct model still contains
\[
\binom{33}{3}+\binom{33}{5}=5~456+237~336=242~792
\]
candidate $K_3$- and $K_5$-blocks. Section~\ref{sec:direct-benchmarks} records one-hour single-thread benchmarks on three monolithic encodings of exactly this reduced instance. Gurobi times out on a $242{,}792$-variable feasibility ILP without finding any incumbent, the direct DLX search times out after visiting millions of nodes, and a deliberately direct SAT encoding already expands to $2{,}773{,}842$ variables and $7{,}593{,}150$ clauses before the solver starts. So the obstruction is not that a direct model cannot be written down; it can. The obstruction is that the unreduced search space remains too large, too symmetric, and too uninformative for a one-shot proof strategy.

This is exactly the regime where a progressive filtering strategy becomes useful: the point is not to replace ILP, DLX, or SAT, but to place them late in the pipeline, when the branch has already been squeezed into a much smaller exact subproblem.

\section{A warm-up: $K_{18}$ and $K_{19}$ with block sizes $3$ and $4$}

The cases $K_{18}$ and $K_{19}$ for exact $\{K_3,K_4\}$-covers are small enough to describe in some detail, but already display the same philosophy as the $K_{33}$ search: first isolate a tiny family of structured residues by congruence and counting, then call ILP only on those residues.

Let $\alpha$ and $\beta$ denote the numbers of $K_3$'s and $K_4$'s in a decomposition of $K_v$. Since
\[
3\alpha+6\beta=\binom v2,
\]
minimising $\alpha+\beta$ is equivalent to minimising $\alpha$. Thus it is enough to push the lower bound on the number of triangles up to the smallest value for which an explicit decomposition exists. Using the standard exact-multiplicity ILP model from Section~2 with Gurobi \cite{gurobi}, one finds decompositions with
$(\alpha,\beta)=(15,18)$ for $K_{18}$ (see Table \ref{tab:k18byk3k4}),
and
$(\alpha,\beta)=(13,22)$ for $K_{19}$ (see Table \ref{tab:k19byk3k4}).

\begin{table}
	\centering
	\footnotesize
	\begin{tabular}{lllllll}
		1 2 10 17 & 2 4 5 7 & 5 9 13 18 &  & 1 14 18 & 3 14 17 & 7 10 18 \\
		1 3 5 15 & 2 6 14 16 & 5 11 16 17 &  & 2 8 18 & 4 8 15 & 7 12 17 \\
		1 4 9 16 & 2 9 12 15 & 6 8 9 17 &  & 3 4 6 & 4 10 11 & 7 15 16 \\
		1 6 7 13 & 3 12 16 18 & 6 11 15 18 &  & 3 7 8 & 4 17 18 & 10 14 15 \\
		1 8 11 12 & 4 12 13 14 & 7 9 11 14 &  & 3 9 10 & 5 8 14 & 13 15 17 \\
		2 3 11 13 & 5 6 10 12 & 8 10 13 16 &  &  &  & 
	\end{tabular}
	\caption{$K_{18}$ decomposed into 15 $K_3$'s and 18 $K_4$'s.}
	\label{tab:k18byk3k4}
\end{table}
\begin{table}
	\centering
	\footnotesize
	\begin{tabular}{llllllll}
		1 8 9 10 & 2 9 11 16 & 4 5 9 18 & 7 8 14 18 &  & 1 2 3 & 3 7 12 & 7 11 17 \\
		1 11 12 13 & 2 10 12 18 & 4 12 16 19 & 7 9 13 19 &  & 1 4 6 & 4 8 11 &  \\
		1 14 15 16 & 3 4 10 17 & 5 8 16 17 & 9 12 15 17 &  & 1 5 7 & 4 13 14 &  \\
		1 17 18 19 & 3 6 9 14 & 5 10 13 15 & 10 11 14 19 &  & 2 8 13 & 5 12 14 &  \\
		2 4 7 15 & 3 8 15 19 & 6 7 10 16 &  &  & 2 14 17 & 6 8 12 &  \\
		2 5 6 19 & 3 13 16 18 & 6 11 15 18 &  &  & 3 5 11 & 6 13 17 & 
	\end{tabular}
	\caption{$K_{19}$ decomposed into 13 $K_3$'s and 22 $K_4$'s.}
	\label{tab:k19byk3k4}
\end{table}

We now focus on the lower bounds.
\begin{theorem}\label{thm:k18k19-warmup}
We have
$C^\xi(18,\{3,4\})=33$, and
$C^\xi(19,\{3,4\})=35$.
\end{theorem}
Equivalently, we are going to show that the minimum possible $\alpha$ is 15 for $K_{18}$ and 13 for $K_{19}$. A detailed version of the proof featuring illustrative figures is in \cite{k18k19paper}. 

\subsection{Lower bound for $K_{18}$}

Fix a $\{K_3,K_4\}$-decomposition of $K_{18}$. For a vertex $x$, let $\alpha_x$ and $\beta_x$ denote the numbers of triangles and quadruples through $x$. Counting the degree of $x$ gives
\[
2\alpha_x+3\beta_x=17,
\]
so $\alpha_x\equiv 1\pmod 3$, and
$\alpha_x\in\{1,4,7\}$.
Moreover $3\alpha+6\beta=\binom{18}{2} =153$, hence $\alpha$ is odd, and 
$$\alpha= \frac{\sum_x \alpha_x}{3} \ge \frac{1\cdot 18}{3}=6,$$
so $\alpha\ge 7$.

\paragraph{Excluding $\alpha=7$.}
If some vertex $u$ satisfies $\alpha_u=7$, then
\[
\alpha=\frac{1}{3}\sum_x \alpha_x\ge \frac{7+17\cdot 1}{3}=8,
\]
a contradiction. Thus all vertices satisfy $\alpha_x\in\{1,4\}$, and since $\sum_x\alpha_x=3\alpha=21$ we deduce that there is exactly one vertex with $\alpha_x=4$. Hence the seven triangles consist of four triangles through one distinguished vertex together with three disjoint triangles. The residual graph after removing those seven triangles is then uniquely determined up to isomorphism, and an ILP completion test shows that it admits no $K_4$-decomposition.

\paragraph{Excluding $\alpha=9$.}
Again $\alpha_x\in\{1,4,7\}$. A vertex with $\alpha_x=7$ is impossible by a simple vertex-count argument: one quickly exceeds the available $18$ vertices whether or not a second heavy vertex shares a triangle with it. Therefore every vertex lies in one or four triangles, and since $\sum_x\alpha_x=27$ there are exactly three vertices $u_1,u_2,u_3$ with $\alpha_{u_i}=4$. These three vertices cannot all lie in one triangle, because that would already force at least $21$ vertices. Hence each pair among them must share a triangle, which yields a unique $9$-triangle pattern on $18$ vertices, namely three pairwise intersecting stars completed by three disjoint outer triangles. An ILP test on the remaining edges again shows that no $K_4$-decomposition exists.

\paragraph{Excluding $\alpha=11$.}
The same degree-congruence argument first rules out $\alpha_x=7$, so every vertex again lies in one or four triangles. Since $\sum_x\alpha_x=33$, there are exactly five vertices $w_1,\dots,w_5$ with $\alpha_{w_i}=4$. One can prove that some triple among these five vertices must itself be a triangle in the design: if no such triple existed, then the number $t$ of pairs among $w_1,\dots,w_5$ sharing a triangle would satisfy $t\ge 9$, so only two intersection patterns remain, and both are rejected by ILP. Thus, after relabelling, there is a triangle $T=\{w_1,w_2,w_3\}$.

For $i=1,2,3$, let $N_i$ be the set of vertices outside $T$ that share a triangle with $w_i$. Then $|N_i|=6$. Since $w_4$ and $w_5$ each lie in at least one triangle not meeting $T$, we have $|N_1\cup N_2\cup N_3|\le 14$. Inclusion--exclusion therefore yields
\[
|N_1\cap N_2|+|N_1\cap N_3|+|N_2\cap N_3|-|N_1\cap N_2\cap N_3|\ge 4.
\]
A short case analysis now forces
\(
N_1\cap N_2\cap N_3=\{w_4,w_5\}
\).
Consequently the $11$ triangles are uniquely determined up to isomorphism: besides $T$ they are the six triangles joining each of $w_4,w_5$ to one of $w_1,w_2,w_3$, together with one final triangle through $w_4$ and $w_5$. The residual graph is again infeasible for a $K_4$-decomposition by ILP.

\paragraph{Excluding $\alpha=13$.}
Now each vertex still satisfies $\alpha_x\in\{1,4\}$, and since $\sum_x\alpha_x=39$, exactly seven vertices $W=\{w_1,\dots,w_7\}$ satisfy $\alpha_{w_i}=4$. Let $t_i$ denote the number of triangles containing exactly $i$ vertices of $W$. Then
\[
t_0+t_1+t_2+t_3=13,
\qquad
3t_3+2t_2+t_1=28,
\qquad
3t_0+2t_1+t_2=11.
\]
Together with the obvious bounds on $t_i$, and the fact that the $K_4$'s must cover at least two edges inside $K_7[W]$, this leaves only three possibilities for $(t_0,t_1,t_2,t_3)$:
\[
(t_0,t_1,t_2,t_3)\in \{(0,2,7,4),\ (0,1,9,3),\ (0,0,11,2)\}.
\]
These are the cases labelled $(C5)$, $(C6)$ and $(C7)$ in the long version of the proof.

If $t_3=4$, then the complement in $K_7[W]$ of the four internal triangles is an even graph of size $9$. Up to isomorphism there are six such graphs; a quick triangle-decomposability test discards three immediately, and the remaining three produce only a handful of placements for the two $t_1$-triangles. Each resulting residual graph is rejected by ILP. If $t_3=3$, the same idea applies to the six even graphs of order $7$ and size $12$; only three survive the first screen, and each yields a finite list of partial triangle configurations, all infeasible by ILP search. Finally, if $t_3=2$, the two internal triangles are either disjoint or intersect in one vertex, and in both cases the remaining eleven mixed triangles are forced up to a tiny number of symmetric choices; all of them fail the ILP test for $K_4$-decomposition.

Therefore $\alpha\notin\{7,9,11,13\}$, so $\alpha\ge 15$. Since the explicit ILP construction has $\alpha=15$, we obtain $C^\xi(18,\{3,4\},2)=15+18=33$.

\subsection{Lower bound for $K_{19}$}

Now consider a $\{K_3,K_4\}$-decomposition of $K_{19}$. Counting the degree of a vertex $x$ gives
\[
2\alpha_x+3\beta_x=18,
\]
so \(\alpha_x\equiv 0\pmod 3\).
In particular, if $\alpha\le 11$ then every vertex satisfies $\alpha_x\in\{0,3,6\}$. Moreover $3\alpha+6\beta=\binom{19}{2} =171$, hence $\alpha$ is odd. So there exists some $x_0\in V(K_{19})$ with $\alpha_{x_0}>0$, and thus $\alpha_{x_0}\ge 3$. These 3 triangles at $x_0$ contain 7 vertices in total, denoted by $\{x_0,x_1,\cdots,x_6\}$, and for $i\in \{0,1,\cdots,6\}$, we have $\alpha_{x_i}>0$, and thus $\alpha_{x_i}\ge 3$. Therefore,
\begin{align*}
	\alpha\ge \frac{3\cdot 7 + 0\cdot(19-7)}{3}=7.
\end{align*}

\paragraph{Excluding $\alpha=7$.}
If some vertex lay in six triangles, then its twelve triangle-neighbours would each lie in at least three triangles, giving
\[
\alpha\ge \frac{6+12\cdot 3}{3}=14,
\]
which is impossible. Hence every triangle-vertex, i.e. every vertex $x$ with $\alpha_x>0$, lies in exactly three triangles, and there are exactly seven such vertices. Their seven triangles therefore form a $2$-$(19,\{4,7^*\},1)$ design, that is, a $\operatorname{PBD}(19,\{4,7^*\},1)$; this is impossible by the existence criterion for incomplete block designs of block size four due to Rees and Stinson \cite{ReesStinson1989}. Thus $\alpha\ne 7$.

\paragraph{Excluding $\alpha=9$.}
Again every triangle-vertex lies in exactly three triangles, so there are exactly nine such vertices. Let $H$ be the union of the nine triangles on this $9$-vertex support. Then $H$ is $6$-regular, so its complement is $2$-regular. Hence the complement is one of
\[
C_9,\qquad C_6\cup C_3,\qquad C_5\cup C_4,\qquad 3C_3.
\]
These are the only four isomorphism types. For each type, the residual graph that should be decomposed into $K_4$'s is built explicitly and tested by ILP; all four are infeasible. Therefore $\alpha\ne 9$.

\paragraph{Excluding $\alpha=11$.}
The same argument shows that every triangle-vertex lies in exactly three triangles, so there are exactly eleven such vertices. Let $H$ be the union of the eleven triangles on this support. Then $H$ is a $6$-regular graph of order $11$. Using \texttt{nauty} \cite{McKayPiperno2014}, one enumerates all $266$ isomorphism types of such graphs \cite{nauty-11-reg6}. For each type, one asks whether the complement of $H$ in $K_{19}$ is decomposable into $K_4$'s. Every instance is rejected by ILP. Hence $\alpha\ne 11$.

We conclude that $\alpha\ge 13$, and the explicit construction with $(\alpha,\beta)=(13,22)$ gives
$C^\xi(19,\{3,4\},2)=13+22=35$.

This completes the proof of Theorem~\ref{thm:k18k19-warmup}. The important methodological point is not the particular small numbers, but the pattern: arithmetic and local degree conditions collapse the search to a tiny set of combinatorial residues, and only those residues are passed to the solver. The $K_{33}$ proof uses exactly the same philosophy, but with several more layers between the first reduction and the final exact certification.

\paragraph*{Why the $K_{33}$ proof does not follow the same ILP pattern?}

At first sight one may ask why the $K_{33}$ case is not handled in exactly the same way as $K_{18}$ and $K_{19}$, namely by a direct ILP formulation combined with a short structural case split. The answer is that the smaller proofs are already much less ``pure ILP'' than such a description suggests. In the $K_{18}$ and $K_{19}$ arguments, the decisive work is done before the solver is called: congruence conditions, local degree restrictions, and finite intersection patterns reduce the search to a very small set of residual graphs. ILP is then used only as a final exact completion test for those sharply reduced residues.

For $K_{33}$ this one-step reduction no longer occurs. Even after the arithmetic reduction to $(\alpha,\beta,\gamma)=(6,0,51)$, one still faces a huge family of candidate $K_5$-blocks together with substantial symmetry and many hard unsatisfiable branches. A direct ILP model can certainly be written down, but it is not guided toward the relevant local structure and gives little information in the branches that matter most. The same applies, in different ways, to monolithic DLX and SAT encodings.

Our method should therefore be viewed not as a dismissal of the approach for $K_{18}$ and $K_{19}$, but as its large-scale continuation. We still use theory first and exact solvers later; the difference is that for $K_{33}$ one round of reduction is insufficient, so the reduction must be iterated through several layers. SAT enters only after symmetry, local signatures, arithmetic profiles, and type-$0$ geometry have already compressed the branch to a much smaller Boolean feasibility problem.

\section{Background for $K_{33}$}

We now turn to mixed coverings by $K_3$, $K_4$, and $K_5$, which is the main topic of our study \cite{k3k4k5paper} on covering with mixed cliques. When we only consider coverings by $K_3$ and $K_4$, both with size divisible by 3, it is easy to observe that the minimum excess cannot be zero in some cases. However, when we introduce $K_5$ along with $K_3$ and $K_4$, it is known that the minimum excess in a $\{K_3,K_4,K_5\}$-covering is always empty except for two tiny exceptional cases $K_6$ and $K_8$ \cite{k3k4k5paper}. The general results established in the broader study of $\{K_3,K_4,K_5\}$-coverings imply the following facts that specialise to $v=33$.

\begin{theorem}[\cite{k3k4k5paper}]\label{thm:general-k33}
For $v\equiv 13\pmod{20}$, the minimum excess in a $\{K_3,K_4,K_5\}$-cover of $K_v$ is empty, and
\[
\Cxi(v,\{3,4,5\},2)\ge \frac{v^2-v+84}{20}.
\]
\end{theorem}
For $v=33$, this implies $\Cxi(33,\{3,4,5\},2)\ge 57$. The minimum excess for $K_{33}$ is clearly empty, since a Steiner triple system of order 33 exists. 

Thus a $57$-block zero-excess decomposition would be extremal. Solving
\[
\alpha+\beta+\gamma=57,
\qquad
3\alpha+6\beta+10\gamma=\binom{33}{2}=528,
\]
we get only two nonnegative integer solutions:
\[
(\alpha,\beta,\gamma)=(2,7,48)
\qquad\text{or}\qquad
(\alpha,\beta,\gamma)=(6,0,51).
\]
The first case is already excluded by the general $v\equiv 13\pmod{20}$ analysis in the mixed-covering paper \cite{k3k4k5paper}. 
\begin{proposition}[\cite{k3k4k5paper}]\label{prop:k3k4k5-v13-a2-b7}
	Let $v\equiv 13\pmod{20}$. Then there is no decomposition of $K_v$ into $2$ $K_3$'s, $7$ $K_4$'s and $(v^2-v-96)/20$ $K_5$'s. 
\end{proposition}
Hence the only remaining $57$-block case is
$(\alpha,\beta,\gamma)=(6,0,51)$,
that is, a decomposition of $K_{33}$ into $51$ copies of $K_5$ and $6$ copies of $K_3$.
This is the case studied algorithmically in the rest of the paper.

\begin{remark}
A later short counting argument also excludes the $58$-block possibilities, so the current numerical lower bound for order 33 is in fact $59$. We include that short argument at the end, but the present paper is centred on the algorithmic exclusion of the $57$-block case, since that is where the new search framework enters.
\end{remark}

\section{Direct formulations and benchmark protocol}\label{sec:direct-benchmarks}

Before developing the layered proof, it is useful to record what actually happens when one insists on a monolithic formulation. The purpose of these experiments is not to prove nonexistence by timeout. Rather, they document that the raw $K_{33}$ instance is too large and too symmetric for a one-shot exact attack, and hence justify the progressive strategy adopted later.

We benchmark the unique $57$-block arithmetic case $(\alpha,\beta,\gamma)=(6,0,51)$, so we already freeze $\beta=0$ and use only candidate $K_3$- and $K_5$-blocks. This still gives
\[
\binom{33}{3}+\binom{33}{5}=5~456+237~336=242~792
\]
candidates in total. All three runs were carried out on the same workstation with an AMD Ryzen~7~5700G CPU under Ubuntu~24.04.4 with a common cutoff of $3600$ seconds and a single-thread policy. We implemented the following three deliberately direct formulations.
\begin{itemize}[leftmargin=1.8em]
\item \textbf{ILP}. One binary variable is used for each candidate $K_3$ or $K_5$. The model contains one equality for each of the $528$ edges of $K_{33}$, one equation fixing the number of triangles, one equation fixing the number of $K_5$'s, and $33$ redundant vertex-degree equalities.
\item \textbf{DLX}. We use one row for each candidate $K_3$ or $K_5$ and one primary column for each edge. The exact triangle quota is enforced during the search, so this is a direct exact-cover search with a fixed number of triples.
\item \textbf{SAT}. We use one Boolean variable for each candidate $K_5$, together with six labelled triangle slots, each slot choosing exactly one of the 5,456 possible triangles. Exact edge coverage is encoded by CNF exactly-one constraints, and triangles are forbidden to repeat across the six slots. This is intentionally direct and therefore keeps the additional $6!$ symmetry of the labelled slots.
\end{itemize}

Table~\ref{tab:direct-benchmarks} summarises the resulting model sizes and outcomes. The full benchmark data is deposited in \cite{k33-benchmark}. In all three paradigms the run reaches the one-hour cutoff without producing a certificate.

\begin{table}[ht]
\centering
\caption{Monolithic one-hour benchmarks for the $K_{33}$ case $(\alpha,\beta,\gamma)=(6,0,51)$.}
\label{tab:direct-benchmarks}
\small
\begin{tabular}{@{}l p{5.3cm} p{2.9cm} c p{3.0cm}@{}}
\toprule
solver & model size & preprocessing/build & cutoff & outcome \\
\midrule
ILP & $242{,}792$ binary vars, $563$ constraints, $3{,}835{,}568$ nonzeros & build $2.21$s & $3600$s & timeout; no feasible incumbent \\
DLX & $242{,}792$ rows, $528$ primary columns & direct enumeration & $3600$s & timeout; no exact cover \\
SAT & $2{,}773{,}842$ CNF vars, $7{,}593{,}150$ clauses & CNF generation $2.50$s & $3600$s & timeout; no certificate \\
\bottomrule
\end{tabular}
\normalsize
\end{table}

The cutoff statistics are equally instructive. Gurobi explores $879$ branch-and-bound nodes and performs $5{,}021{,}453$ simplex iterations, but never finds even one feasible incumbent. The direct DLX search visits $4{,}677{,}632$ nodes, performs $866{,}460{,}362{,}139$ compatibility checks, and reaches best partial depth $46$, still far from the required $57$ blocks. The SAT encoding is already large at input level: Kissat parses a $146$ MB DIMACS file with header
\[
\texttt{p cnf 2773842 7593150},
\]
and after the full one-hour budget still returns neither \texttt{sat} nor \texttt{unsat}.

These experiments confirm the point of Section~2.4. Even after the first arithmetic reduction to $(6,0,51)$, the direct search space is still enormous and highly symmetric. Monolithic ILP, DLX, and SAT formulations are certainly possible, but by themselves they do not expose the local structure responsible for infeasibility. This is why the remainder of the paper replaces one large exact search by a sequence of smaller exact filters.

\section{The layered search for the $57$-block case}

We now describe the actual search. Assume that $K_{33}$ is decomposed into $51$ copies of $K_5$ and $6$ copies of $K_3$.
\subsection{The six triangles live on a distinguished $K_9$}
Checking degree of any vertex $x$:
$$32 = 2\alpha_x + 4\gamma_x,$$
we get $\alpha_x\equiv 0\pmod{2}$. 
We now assert that there is no vertex $x$ with $\alpha_x\ge 4$. Otherwise, there are at least 8 vertices $y$ other than $x$ with $\alpha_y>0$, hence
$$\alpha\ge \frac{1}{3}\left(4 + 2\cdot 8\right)>6,$$
a contradiction!
So we have exactly 9 vertices each contained in 2 triples, and the remaining part is decomposed into quintuples. 

Build an auxiliary graph $G$ whose vertices are those 6 triangles in the decomposition, and two triangles are joined by an edge if and only if they share a vertex. Clearly, $G$ is a cubic graph of order 6. Hence $\overline{G}\in \{C_6, 2C_3\}$. So $G\in\{\operatorname{GP}(3,1),K_{3,3}\}$, where $\operatorname{GP}(3,1)$ is the triangular prism graph. 

So without loss of generality, we have just two possible patterns of the 6 triangles on 9 vertices: 
\begin{itemize}
	\item $L(K_{3,3})=K_3\square K_3$: $(1,2,3), (4,5,6), (7,8,9), (1,4,7), (2,5,8), (3,6,9)$. We shall call it the grid pattern.
	\item $L(\operatorname{GP}(3,1))$: $(1,3,7), (1,2,8), (2,3,9), (4,6,7), (4,5,8), (5,6,9)$. We shall call it the prism pattern.
\end{itemize}

A useful observation is about the automorphism of the two triangle patterns on 9 vertices:
\begin{lemma}\label{lem:aut-t}
	$\operatorname{Aut}(K_3\square K_3)=\operatorname{Aut}(K_{3,3})=(\mathfrak S_3\times \mathfrak S_3)\rtimes {\mathbb Z/2\mathbb Z}$, and $\operatorname{Aut}(L(\operatorname{GP}(3,1)))=\operatorname{Aut}(\operatorname{GP}(3,1))=\operatorname{Dih}_6={\mathbb Z/2\mathbb Z}\times {\mathfrak S_3}$.
\end{lemma}
\begin{proof}
	We will omit the rigorous algebraic proof. For $K_{3,3}$, every automorphism can permute vertices inside each bipartition class, as well as swap the 2 classes. For the prism $\operatorname{GP}(3,1)$, every automorphism can permute vertices on both layers simultaneously, as well as swap the 2 triangular layers. The lemma then follows from the following version of Whitney's line graph theorem.
\end{proof}
\begin{theorem}[\cite{Whitney_1932, Hemminger_1972}]
	If \(G\) is a finite simple connected graph with at least \(3\) vertices, then the natural homomorphism
	\[
	\mathrm{Aut}(G)\to \mathrm{Aut}(L(G))
	\]
	is an isomorphism.
\end{theorem}
\subsection{Type-$s$ quintuples and the parameter $n_3$}
Let $T$ denote the union of these 6 triangles, notice that the clique number $\omega(K_9-T)=3$ in both cases. So in the decomposition of $K_{33}$, there are no $K_5$'s with at least 4 vertices from $V(T)$. Denote by $n_s$ the number of $K_5$'s with exactly $s$ vertices from $V(T)$ (and $5-s$ vertices from $V(K_{33})\setminus V(T)$), and we call such $K_5$'s the type-$s$ quintuples. Notice that only type-0, type-1, type-2, type-3 quintuples exist. Then
\begin{align}\label{eq:k33-1}
	n_2+3n_3=e(K_9)-6\cdot 3=18\\
	n_1 + 2n_2 +3n_3 = 7\cdot 9 = 63 \\
	n_0+n_1+n_2+n_3 = 51.
\end{align}
And there cannot be more independent equations. Clearly $n_3\in\{0,1,\cdots,6\}$ by \eqref{eq:k33-1}, and the solution is
\begin{align}\label{eq:k33-4}
	(n_0,n_1,n_2,n_3)=(6-n_3,27+3n_3,18-3n_3,n_3).
\end{align}
In fact, we can also show:
\begin{lemma}
	$n_3\le 4$ in both the grid and the prism patterns.
\end{lemma}
\begin{proof}
	Consider any vertex $x\notin V(T)$, then $\gamma_x=8$. Let $n_s(x)$ denote the number of type-$s$ $K_5$'s containing $x$, then
	\begin{align}\label{eq:nsx-1}
		n_0(x)+n_1(x)+n_2(x)+n_3(x)=8\\
		n_1(x) +2n_2(x) +3n_3(x)=9 \label{eq:nsx-2}
	\end{align}
	and we have 11 possible solutions of $(n_0(x),n_1(x),n_2(x),n_3(x))$, as shown in Table \ref{tab:k33-local}.
	
	\begin{table}[]
		\centering
		\caption{Possible solutions for local profiles}
		\label{tab:k33-local}
		\begin{tabular}{cccc}
			$n_0(x)$ & $n_1(x)$ & $n_2(x)$ & $n_3(x)$ \\ \hline
			0 & 7&1&0\\
			1&5&2&0\\
			1&6&0&1\\ 
			2&3&3&0\\
			2&4&1&1\\
			3&1&4&0\\
			3&2&2&1\\
			3&3&0&2\\ 
			4&0&3&1\\ 
			4&1&1&2\\
			5&0&0&3
		\end{tabular}
	\end{table}
	
	Suppose $m$ vertices outside $V(T)$ have $n_0(x)=0$, 
	then by \eqref{eq:k33-4},
	$$24-m\le \sum_{x\notin K_9} n_0(x)=5n_0 = 30-5n_3,$$
	so $m\ge 5n_3-6$. And since these $m$ vertices all have $n_2(x)=1$, by \eqref{eq:k33-4},
	$$m\le \sum_{x\notin K_9}n_2(x)=3n_2=54-9n_3,$$
	so $5n_3-6\le 54-9n_3$, and we get $n_3\le 4$. 
\end{proof}

So globally, the distribution of $K_5$'s in \eqref{eq:k33-4} has 5 possibilities: 
\begin{align}
	(n_0,n_1,n_2,n_3)=(6,27,18,0),(5,30,15,1),(4,33,12,2),(3,36,9,3),(2,39,6,4).
\end{align}

\subsection{The philosophy of the layered search}

The layered method has the following generic form. We shall define and explain the terms in detail in Section~6.4 using the prism $n_3=2$ case. 
\begin{enumerate}[leftmargin=1.8em]
\item Fix the triangle pattern (grid or prism) on the distinguished $K_9$ and, when relevant, work up to the automorphism group of that pattern.
\item Enumerate admissible local signatures of the outside vertices. A local signature records how an outside vertex can participate in type-$2$ and type-$3$ quintuples relative to the fixed $K_9$-skeleton, and satisfies one of the local patterns in Table~\ref{tab:k33-local}.
\item Assemble these local signatures into global profile distributions satisfying the exact counting equations such as \eqref{eq:sum-nix}.
\item Apply a type-$0$ geometry filter concerning how the $K_5$'s are placed outside the distinguished $K_9$. This is a purely combinatorial test that discards global distributions that cannot be realised by the required number of type-$0$ quintuples.
\item Only now form SAT instances for the surviving branches.
\item Decode each satisfying assignment back into the covering language and perform a separate edge-disjointness audit. If some edges are covered twice, the branch fails the verification and should be discarded.
\item If a tiny residual branch survives, fix the already determined structure and solve only the remaining problem of decomposition into type-1 $K_5$'s.
\end{enumerate}

Thus SAT is not used as a whole proof. It is only one layer in a larger exact pipeline.

\subsection{A representative standard case: the prism pattern with $n_3=2$}
We illustrate our computational proof scheme in detail using the prism case
\((n_0,n_1,n_2,n_3)=(4,33,12,2)\).
This is the most representative standard case: it contains a nontrivial split into two $K_9$-orbits, a full arithmetic stage, an exact type-$0$ geometry filter, a SAT realisation stage, as well as a decoding-and-audit stage.

	Recall that in this case the 6 triangles in the distinguished $K_9$ are
	\(137,\ 128,\ 239,\ 467,\ 458,\ 569\).
	The uncovered $K_9$-edges are precisely those not lying in these 6 triangles.  A type-$3$ $K_5$ associates 2 vertices outside $V(T)$ with a $3$-subset of $V(T)$ all of whose 3 edges are not in $E(T)$, hence we may use the terms \emph{type-3 triples}. The set of all type-3 triples on the distinguished $K_9$ is denoted by $\mathcal T_3$. 
	
	\paragraph*{Identification of the \(K_9\)-orbits.}
	The group \(\operatorname{Aut}(T)\) acts naturally on the vertex set \(\{1,\dots,9\}\), hence also on all \(3\)-subsets of \(\{1,\dots,9\}\), and therefore on all admissible families of type-\(3\) triples.
	Two such families which lie in the same \(G\)-orbit are equivalent for our search, since they lead to isomorphic residual instances.
	It is therefore sufficient to choose one representative from each orbit.
	
	In the present case \(n_3=2\), we first enumerate all triples on the distinguished \(K_9\) that can serve as type-\(3\) blocks while remaining compatible with the fixed prism pattern.
	This yields exactly 4 legal type-\(3\) triples: 149, 257, 368, 789.
	Since \(n_3=2\), we must choose two of these four triples, so there are initially
	\(\binom{4}{2}=6\)
	raw candidates. 
	To determine the orbits, we let the automorphism group from Lemma \ref{lem:aut-t} act on the six \(2\)-subsets of the 4 legal type-\(3\) triples. The resulting orbit partition is displayed in Table~\ref{tab:action-orbit-prism-2n3}. A convenient distinguishing invariant is the multiplicity profile of the nine vertices under the 2 chosen triples.
	\begin{table}[]
		\centering
		\begin{tabular}{ccl}
			\toprule
			pair of type-\(3\) triples & multiplicity of vertices & orbit \\
			\midrule
			149, 257 & \(1,1,0,1,1,0,1,0,1\) & \(\mathcal O_1\) \\
			149, 368 & \(1,0,1,1,0,1,0,1,1\) & \(\mathcal O_1\) \\
			257, 368 & \(0,1,1,0,1,1,1,1,0\) & \(\mathcal O_1\) \\
			\midrule
			149, 789 & \(1,0,0,1,0,0,1,1,2\) & \(\mathcal O_2\) \\
			257, 789 & \(0,1,0,0,1,0,2,1,1\) & \(\mathcal O_2\) \\
			368, 789 & \(0,0,1,0,0,1,1,2,1\) & \(\mathcal O_2\) \\
			\bottomrule
		\end{tabular}
		\caption{Action of \(\mathrm{Aut}(T)\) on the 6 pairs of legal type-\(3\) triples in the prism case \(n_3=2\).}
		\label{tab:action-orbit-prism-2n3}
	\end{table}
	The 3 pairs in the first block form one orbit, and the 3 pairs in the second block form the other. In particular, the 6 raw candidates split into exactly 2 \(\mathrm{Aut}(T)\)-orbits, represented by $\mathcal O_1=\{\{1,4,9\},\{2,5,7\}\}$ and $\mathcal O_2=\{\{1,4,9\},\{7,8,9\}\}$. We now analyse the 2 orbits separately.

	\paragraph*{Orbit 1: local signatures.}
	Fix the 2 type-3 triples 149 and 257 in the first orbit, then there are 12 leftover pairs inside the distinguished $K_9$: 15, 16, 24, 26, 34, 35, 36, 38, 68, 78, 79, 89. They must all be covered by some $K_5$'s of type 2, so we call them \emph{type-2 pairs}. The set of all type-2 pairs on the distinguished $K_9$ is denoted by $\mathcal P_2$. For a single outside vertex $x\notin V(T)$, a \emph{local signature} $\sigma$ records which leftover type-$2$ pairs and which fixed type-$3$ triples share a quintuple with $x$, hence $\sigma = (\pi_\sigma, \tau_\sigma)\in \mathcal P_2\times \mathcal T_3$.
	Clearly, in one local signature for $x$, all pairs and triples must be pairwise disjoint. The vertices in $K_9$ that are not present in a local signature for $x$ must each be joined to $x$ via a type-1 $K_5$, and we call them the \emph{type-1 singles for $x$}.
	Once a local signature is chosen for $x$, $n_2(x)$ and $n_3(x)$ are known, and we have 
	\[
	n_1(x)=9-2n_2(x)-3n_3(x),
	\qquad
	n_0(x)=8-n_1(x)-n_2(x)-n_3(x)
	\]
	from \eqref{eq:nsx-1} and \eqref{eq:nsx-2}. Since $n_3(x)\le 2$, we can require $n_0(x)\le 4$ (see Table \ref{tab:k33-local}). In other words, there is a map from the set of local signatures, denoted by $\Sigma$, to the set of local profiles, denoted by $\Pi$: 
	$$\nu: \Sigma\to \Pi, \qquad \sigma\mapsto (n_0(x),n_1(x),n_2(x),n_3(x)).$$
	We now enumerate all possible local signatures for a vertex outside the distinguished $K_9$. This is done by the following steps: 
	\begin{enumerate}
		\item choose a subset of the 2 fixed type-3 triples, either none of them, or one of them, or both of them; 
		\item for each such subset, enumerate every subset of those 12 leftover type-2 pairs disjoint from it; 
		\item keep only those choices in which the selected type-2 pairs are pairwise disjoint; 
		\item compute the resulting local profile $(n_0(x),n_1(x),n_2(x),n_3(x))$ and discard the choice if $n_1(x)<0$, $n_0(x)<0$ or $n_0(x)>4$. 
	\end{enumerate}
	
	The search yields exactly $151$ local signatures for orbit~1. We show an excerpt below.
	
	\small
	\begin{verbatim}
		[1]  pairs=[[1,5]], profile=(0,7,1,0), triples=[], singles=[2,3,4,6,7,8,9]
		[2]  pairs=[[1,6]], profile=(0,7,1,0), triples=[], singles=[2,3,4,5,7,8,9]
		[3]  pairs=[[2,4]], profile=(0,7,1,0), triples=[], singles=[1,3,5,6,7,8,9]
		[4]  pairs=[[2,6]], profile=(0,7,1,0), triples=[], singles=[1,3,4,5,7,8,9]
		[5]  pairs=[[3,4]], profile=(0,7,1,0), triples=[], singles=[1,2,5,6,7,8,9]
		...
		[147] pairs=[[1,6],[3,4],[8,9]], profile=(4,0,3,1), triples=[[2,5,7]], singles=[]
		[148] pairs=[[2,6],[3,5],[7,8]], profile=(4,0,3,1), triples=[[1,4,9]], singles=[]
		[149] pairs=[[3,6]], profile=(4,1,1,2), triples=[[1,4,9],[2,5,7]], singles=[8]
		[150] pairs=[[3,8]], profile=(4,1,1,2), triples=[[1,4,9],[2,5,7]], singles=[6]
		[151] pairs=[[6,8]], profile=(4,1,1,2), triples=[[1,4,9],[2,5,7]], singles=[3]
	\end{verbatim}
	\normalsize
	\paragraph*{Orbit 1: global distributions.}
	The next step is a simple arithmetic stage to enumerate all possible combinations of the local profiles $(n_0(x),n_1(x),n_2(x),n_3(x))$ for the 24 vertices $x\notin V(T)$, so that the global counting equations coming from $(n_0,n_1,n_2,n_3)=(4,33,12,2)$ are satisfied. More specifically, if $c_p$ denotes the multiplicity of the profile type $p=(p_0,p_1,p_2,p_3)$, i.e. $p_s=n_s(x)$ for $0\le s\le 3$, then we require
	\begin{align}\label{eq:sum-nix}
		\sum_p c_p = 24,
		\qquad 
		\sum_p c_p p_0 = 5n_0,
		\qquad
		\sum_p c_p p_1 = 4n_1,
		\qquad
		\sum_p c_p p_2 = 3n_2,
		\qquad
		\sum_p c_p p_3 = 2n_3.
	\end{align}
	These are exactly the totals obtained by summing $n_i(x)$ over all 24 vertices outside the distinguished $K_9$. 
	
	The tuple $(c_p)_p$ where $p$ traverses the list of possible local patterns (see Table \ref{tab:k33-local}) gives a \emph{global distribution} of local profiles. We enumerate all possible global distributions in the following way:
	\begin{enumerate}
		\item let $P$ be the finite list of profile types appearing among the local signatures, hence $|P|\le 11$ (cf. Table \ref{tab:k33-local}); 
		\item recursively assign a multiplicity $c_p\ge 0$ to each $p\in P$;
		\item at each step, prune the search if any partial sum already exceeds one of the target totals $(24,20,132,36,4)$ in \eqref{eq:sum-nix}.
	\end{enumerate}
	This way, we keep exactly those multiplicity vectors satisfying the 4 global equations above in \eqref{eq:sum-nix}. 
	For orbit~1 this produces exactly $931$ arithmetic distributions.  The following is an excerpt. 
	
	\small
	\begin{verbatim}
		[1] {(0,7,1,0):4, (1,5,2,0):16, (1,6,0,1):4}
		[2] {(0,7,1,0):5, (1,5,2,0):14, (1,6,0,1):4, (2,3,3,0):1}
		[3] {(0,7,1,0):5, (1,5,2,0):15, (1,6,0,1):3, (2,4,1,1):1}
		[4] {(0,7,1,0):6, (1,5,2,0):12, (1,6,0,1):4, (2,3,3,0):2}
		[5] {(0,7,1,0):6, (1,5,2,0):13, (1,6,0,1):3, (2,3,3,0):1, (2,4,1,1):1}
		...
		[927] {(0,7,1,0):18, (2,4,1,1):1, (3,1,4,0):2, (4,0,3,1):3}
		[928] {(0,7,1,0):18, (2,3,3,0):1, (3,1,4,0):1, (3,2,2,1):1, (4,0,3,1):3}
		[929] {(0,7,1,0):18, (2,3,3,0):1, (3,1,4,0):2, (4,0,3,1):2, (4,1,1,2):1}
		[930] {(0,7,1,0):18, (2,3,3,0):2, (4,0,3,1):4}
		[931] {(0,7,1,0):18, (1,5,2,0):1, (3,1,4,0):1, (4,0,3,1):4}
	\end{verbatim}
	\normalsize
	\paragraph*{Orbit 1: type-$0$ geometry filter.}
	When generating global distributions of local profiles, we have not considered whether the type-$0$ incidences can actually be arranged on the 4 type-$0$ quintuples ($n_0=4$).  This is what the type-$0$ geometry filter decides.
	For a given distribution, let
	\[
	a_i = \left|\{x:\ n_0(x)=i\}\right|, \qquad i=0,1,2,3,4.
	\]
	The following filter aims at ensuring no two out of the four type-0 $K_5$'s, denoted by $Q_1,Q_2,Q_3,Q_4$, share more than 1 vertex outside the distinguished $K_9$. We proceed with the following steps:
	\begin{enumerate}
		\item given a global distribution, compute $(a_0,a_1,a_2,a_3,a_4)$; 
		\item check the possible type-0 incidences for vertices with $a_4>0$:
		\begin{itemize}
			\item if $a_4=0$, skip;
			\item if $a_4=1$, then this unique vertex belongs to $Q_1\cap Q_2\cap Q_3\cap Q_4$, thus we require $a_2=a_3=0$:
			\begin{itemize}
				\item if $a_2=a_3=0$, jump to step 5;
				\item otherwise, reject the distribution immediately;
			\end{itemize}
			\item if $a_4\ge 2$, reject immediately;
		\end{itemize}
		\item enumerate all feasible type-0 incidences for vertices with $a_3>0$:  
		\begin{itemize}
			\item if $a_3=0$, skip;
			\item if $a_3=1$, loop over $\{Q_1,Q_2,Q_3\}$, $\{Q_1,Q_2,Q_4\}$, $\{Q_1,Q_3,Q_4\}$, $\{Q_2,Q_3,Q_4\}$; 
			\item if $a_3\ge 2$, reject immediately, for these vertices will be commonly in at least 2 type-0 $K_5$'s; 
		\end{itemize}
		\item enumerate all feasible type-0 incidences for vertices with $a_2>0$ that are compatible with the 3 previously type-0 $K_5$'s, if any:
		\begin{itemize}
			\item if $a_3=0$, choose for each vertex with $a_2>0$ one of the 6 possible pairs of type-0 $K_5$'s: $\{Q_1,Q_2\}$, $\{Q_1,Q_3\}$, $\{Q_1,Q_4\}$, $\{Q_2,Q_3\}$, $\{Q_2,Q_4\}$, $\{Q_3,Q_4\}$; if $a_2>6$, reject immediately;
			\item if $a_3=1$ and without loss of generality $\{Q_1,Q_2,Q_3\}$ is selected in the last step, choose for each vertex with $a_2>0$ one of the 3 possible pairs of type-0 $K_5$'s: $\{Q_1,Q_4\}$, $\{Q_2,Q_4\}$, $\{Q_3,Q_4\}$; if $a_2>4$, reject immediately;
		\end{itemize}
		\item check the number of incident vertices for each type-0 $K_5$, they must not exceed 5 each and sum up to $20-a_1$, otherwise reject.
	\end{enumerate}
	
	For orbit~1, exactly $56$ of the $931$ arithmetic distributions survive this filter.  
	
	\paragraph*{Orbit 1: signature feasibility test with SAT.}
	This step tests the feasibility of local signature assignments with respect to a global distribution that passes the type-0 geometry filter. 
	
	Denote by \(U=V(K_{33})\setminus V(T)\), and
	the type-0 geometry filter yields, for a survived distribution $D=(m_p)_{p\in \Pi^+}$, where $\Pi^+:=\{p\in \Pi:p_0>0\}$, a map
	\[
	\theta:U\to \mathcal P(\{1,2,3,4\}),
	\qquad
	\theta(v)=\{\,i\in\{1,2,3,4\}:v\in Q_i\,\}.
	\]
	Let
	\(U^{+}:=\{v\in U:\theta(v)\neq\varnothing\}\).
	For each \(v\in U^{+}\), let \(\Sigma(v):=\{\sigma\in \Sigma: n_0(v)=|\theta(v)|\}\) be the set of local signatures compatible with \(\theta(v)\).
	
	Introduce Boolean variables
	\(
	x_{v,\sigma}\in\{0,1\}
	\)
	for \(v\in U^{+}\) and \(\sigma\in\Sigma(v)\). One may view a satisfying signature assignment as a function
	\[
	\varphi\in \prod_{v\in U^{+}}\Sigma(v),
	\qquad
	x_{v,\sigma}=1 \iff \varphi(v)=\sigma .
	\]
	
	The CNF encodes the following constraints.
	\begin{enumerate}
		\item Each vertex is assigned only 1 signature: 
		\[(\forall v\in U^{+}) \qquad \sum_{\sigma\in\Sigma(v)} x_{v,\sigma}=1.\]
		\item Each local profile is assigned according to the prescribed multiplicity in the distribution:
		\[(\forall p\in \Pi^+)\qquad \sum_{v\in U^{+}}
		\sum_{\Sigma(v)\cap \nu^{-1}(p)}
		x_{v,\sigma}
		=
		m_p .\]
		\item Each type-3 triple is associated with at most 2 outside vertices:
		\[(\forall \tau\in \mathcal T_3)\qquad 	\sum_{v\in U^{+}}
		\sum_{\sigma\in\Sigma(v)\,:\,\tau_\sigma=\tau}
		x_{v,\sigma}
		\le 
		2.\]
		In fact, in this case it is observed that any local profile with $n_3(x)>0$ satisfies $n_0(x)>0$, so the inequality can be replaced by an equality. 
		\item Each type-2 triple is associated with at most 3 outside vertices: 
		\[(\forall e\in \mathcal P_2)\qquad
		\sum_{v\in U^{+}}
		\sum_{\sigma\in\Sigma(v)\,:\,\pi_\sigma=e}
		x_{v,\sigma}
		\le 3.
		\]
		\item Compatibility constraints induced by type-0 geometry: 
		\begin{align*}
			(\forall v,w\in U^+)\ &(\forall \sigma\in \Sigma(v))\ (\forall \rho\in \Sigma(w))\\
			&\left[ v\neq w\  
			\wedge\  \theta(v)\cap \theta(w)\neq\varnothing\ 
			\wedge\ (\pi_\sigma\cup \tau_\sigma)\cap (\pi_\rho\cup \tau_\rho)\neq \varnothing \right]\ 
			\rightarrow\ x_{v,\sigma}+x_{w,\rho}\le 1.	
		\end{align*}
	\end{enumerate}
	
	The solver used in our research is \texttt{kissat} (see \cite{BiereFazekasFleuryHeisinger2020,KissatWeb}).  Among the $56$ branches that survive exact type-$0$ geometry, the SAT solver returns $42$ satisfiable instances (\texttt{SAT}) and $14$ unsatisfiable ones (\texttt{UNSAT}). The \texttt{UNSAT} instances are immediately excluded.  
	
	\paragraph*{Orbit 1: decoding and edge-disjointness audit.}
	Every satisfiable model (\texttt{SAT}) must now be decoded and checked against the global edge-disjointness conditions. Decoding simply means reading off, for each outside vertex, which signature was selected by the positive variables of the SAT model.  
	The audit then verifies that these chosen signatures do not create forbidden edge overlaps.  Specifically, if two type-$2$ quintuples share two outside vertices, or share one outside vertex together with a $K_9$-vertex, then they would overlap in an edge of $K_{33}$ and the branch must be rejected.  Analogous checks are carried out for type-$2$/type-$3$ $K_5$'s and type-$3$/type-$3$ $K_5$'s interactions.
	
	In orbit~1, every one of the $42$ satisfiable branches fails this audit.  In fact, every audited branch fails for the same reason:
	
	\small
	\begin{verbatim}
		two type-2 blocks share at least two outside vertices.
	\end{verbatim}
	\normalsize
	Thus orbit~1 is completely excluded.
	
	\paragraph*{Orbit 2: same procedures.}
	The second orbit is treated in exactly the same way, so we only record the corresponding numbers.
	The local-signature enumeration yields $134$ local signatures, and the arithmetic stage yields $513$ global distributions.  
	Of these, $46$ survive the type-$0$ geometry filter.
	The SAT stage yields $38$ satisfiable instances and $8$ unsatisfiable ones.
	After decoding and edge-disjointness audit, $35$ of the satisfiable branches fail, again because two type-$2$ blocks share at least two outside vertices, and exactly three survive:
	
	\small
	\begin{verbatim}
		dist 6, dist 8, dist 26
	\end{verbatim}
	\normalsize
	These are the only residual branches in the entire case.
	
	\paragraph*{Orbit 2: final type-1 feasibility check for residual branches.}
	For each of the three residual orbit~2 branches, the type-$0$, type-$2$, and type-$3$ $K_5$'s structure is already fixed by the previous stages.  It remains only to complete the uncovered edges by type-$1$ $K_5$'s.  The final script does exactly this by computing the residual uncovered edges, and checking whether each such edge has an admissible type-$1$ candidate. If no immediate contradiction appears, we encode the edges' incidences with type-1 blocks as a CNF and solve in \texttt{kissat}. 
	
	In the first two cases we get immediate negative results:
	
	\small
	\begin{verbatim}
		dist 6  : immediate_unsat, reason = residual edges 
		(9, 10), (10, 25), (10, 27), (10, 33) have no type-1 candidates
		dist 8  : immediate_unsat, reason = residual edges 
		(3, 10), (10, 24) have no type-1 candidates
	\end{verbatim}
	\normalsize
	And for the last case (\texttt{dist 26}), a CNF is generated and the SAT solver returns \texttt{UNSAT}:
	
	\small
	\begin{verbatim}
		c ---- [ banner ] ------------------------------------------------------------
		c
		c Kissat SAT Solver
		...
		c ---- [ parsing ] -----------------------------------------------------------
		...
		c parsed 'p cnf 36927 100624' header
		c closing input after reading 1581792 bytes (2 MB)
		c finished parsing after 0.01 seconds
		c
		c ---- [ solving ] -----------------------------------------------------------
		c
		c seconds switched rate     size/glue tier1  binary        remaining
		c        MB reductions conflicts size  tier2    irredundant
		c         level restarts redundant glue  trail        variables
		c
		c *  0.01  8 0 0 0  0 0   0   0 0.0 0 0 0 0 0% 90213 320 34960 95%
		c lucky 70 units
		c l  0.01  8 0 0 0  0 0   0   0 0.0 0 0 0 0 0% 90213 320 34890 94%
		c (  0.01  8 0 0 0  0 0   0   0 0.0 0 0 0 0 0% 90213 320 34890 94%
		c .  0.02  8 0 0 0  0 0   0   0 0.0 0 0 0 0 0% 82406 320 34890 94%
		c )  0.05 11 0 0 0  0 0   0   0 0.0 0 0 0 0 0% 82406 320 34890 94%
		c {  0.05 11 0 0 0  0 0   0   0 0.0 0 0 2 6 0% 82406 320 34890 94%
			c i  0.05 11 277 0 0  0 9  56  38 2.1 16 7 2 4 12% 82412 320 34889 94%
			c i  0.05 11 276 0 0  0 9  59  38 2.1 15 7 2 4 12% 82414 320 34888 94%
			c i  0.05 11 269 0 0  0 8  66  36 2.1 13 6 2 4 11% 82416 320 34882 94%
			c i  0.05 11 245 0 0  1 11  91  49 1.7 13 8 2 4 11% 82425 320 34881 94%
			c i  0.06 11 132 0 0 26 10 182 114 1.8 15 8 5 10 8% 82432 320 33407 90%
			c }  0.06 11 131 0 0 26 10 184 114 1.8 15 8 5 10 8% 82432 320 33127 90%
		c 0  0.06 11 131 0 0 26 10 184 114 1.8 15 8 5 10 8% 82432 320 33127 90%
		c
		c ---- [ result ] ------------------------------------------------------------
		c
		s UNSATISFIABLE
		...
		c ---- [ resources ] ---------------------------------------------------------
		c
		c maximum-resident-set-size:         19927040 bytes         19 MB
		c process-time:                                              0.06 seconds
		c
		c ---- [ shutting down ] -----------------------------------------------------
		c
		c exit 20
	\end{verbatim}
	\normalsize
	Hence the three residual orbit~2 branches are all excluded.
	
	Since orbit~1 contributes no audited survivor and orbit~2 contributes only the three branches, all of which are finally eliminated, the prism case $(n_0,n_1,n_2,n_3)=(4,33,12,2)$ is impossible.


The remaining nine cases are handled by the same general philosophy, but two variants occur.

\subsection{The standard cases with $1\le n_3\le 3$}

Six of the ten cases are treated by the same standard pipeline as in the representative \texttt{prism\_2n3} example in Section 6.4. Their numerical profile is recorded in Table~\ref{tab:standard-pipeline}.

\begin{table}[]
\centering
\setlength{\tabcolsep}{3pt}
\renewcommand{\arraystretch}{1.12}
\caption{Numerical profile of the cases treated by the standard pipeline.}
\label{tab:standard-pipeline}
\resizebox{\textwidth}{!}{%
\begin{tabular}{llrrrrrrrr}
\toprule
case & orbit & local sig. & global dist. & type-0 surv. & SAT & UNSAT & audit surv. & hard residual & final surv. \\
\midrule
\texttt{grid\_1n3} & --    & 227 & 441  & 100 & 75  & 25 & 0 & 0 & 0 \\
\midrule
\multirow{3}{*}{\texttt{grid\_2n3}}
& 1     & 130 & 513  & 46  & 38  & 8  & 0 & 0 & 0 \\
& 2     & 143 & 931  & 56  & 42  & 14 & 0 & 0 & 0 \\
& total & 273 & 1444 & 102 & 80  & 22 & 0 & 0 & 0 \\
\midrule
\multirow{3}{*}{\texttt{grid\_3n3}}
& 1     & 77  & 292  & 12  & 9   & 3  & 2 & 0 & 0 \\
& 2     & 91  & 251  & 12  & 10  & 2  & 1 & 0 & 0 \\
& total & 168 & 543  & 24  & 19  & 5  & 3 & 0 & 0 \\
\midrule
\multirow{3}{*}{\texttt{prism\_1n3}}
& 1     & 235 & 441  & 100 & 75  & 24 & 1 & 1 & 0 \\
& 2     & 231 & 441  & 100 & 75  & 24 & 0 & 1 & 0 \\
& total & 466 & 882  & 200 & 150 & 48 & 1 & 2 & 0 \\
\midrule
\multirow{3}{*}{\texttt{prism\_2n3}}
& 1     & 151 & 931  & 56  & 42  & 14 & 0 & 0 & 0 \\
& 2     & 134 & 513  & 46  & 38  & 8  & 3 & 0 & 0 \\
& total & 285 & 1444 & 102 & 80  & 22 & 3 & 0 & 0 \\
\midrule
\multirow{3}{*}{\texttt{prism\_3n3}}
& 1     & 95  & 207  & 13  & 11  & 2  & 2 & 0 & 0 \\
& 2     & 75  & 128  & 12  & 9   & 3  & 0 & 0 & 0 \\
& total & 170 & 335  & 25  & 20  & 5  & 2 & 0 & 0 \\
\bottomrule
\end{tabular}%
}
\end{table}

The only nonzero ``hard residual'' count occurs in the case \texttt{prism\_1n3}. There the standard pipeline leaves one tiny residual branch \texttt{dist 015} in each orbit. Those two branches are then excluded by a direct combinatorial argument based on the following residual data of the branches extracted by the inspection script. 

\small
\begin{verbatim}
	Residual data for prism_1n3, dist 015.
	
	Common distribution (both orbits):
	A - (0,7,1,0): 2
	B - (1,5,2,0): 20
	C - (1,6,0,1): 1
	D - (4,0,3,1): 1
	
	Common type-0 witness (both orbits):
	singleton_counts_per_block = [4,4,4,4,5]
	quad_vertices = [[0,1,2,3]]
	unique block missing from D-support = 4
	
	Orbit 1, dist 015:
	unique C-signature:
	triple  = [1,4,9]
	singles = [2,3,5,6,7,8]
	
	admissible D-signatures:
	triple [1,4,9], pairs {{2,5},{3,6},{7,8}}
	triple [1,4,9], pairs {{2,6},{3,5},{7,8}}
	triple [1,4,9], pairs {{2,6},{3,8},{5,7}}
	triple [1,4,9], pairs {{2,7},{3,5},{6,8}}
	
	Orbit 2, dist 015:
	unique C-signature:
	triple  = [7,8,9]
	singles = [1,2,3,4,5,6]
	
	admissible D-signatures:
	triple [7,8,9], pairs {{1,4},{2,5},{3,6}}
	triple [7,8,9], pairs {{1,4},{2,6},{3,5}}
	triple [7,8,9], pairs {{1,5},{2,4},{3,6}}
	triple [7,8,9], pairs {{1,5},{2,6},{3,4}}
	triple [7,8,9], pairs {{1,6},{2,4},{3,5}}
	triple [7,8,9], pairs {{1,6},{2,5},{3,4}}
\end{verbatim}
\normalsize
We now show that:
\begin{proposition}
	For the prism case $n_3=1$, the residual branches {\rm\texttt{dist 015}} for both $K_9$-orbits are infeasible. As a result, the prism case $(n_0,n_1,n_2,n_3)=(5,30,15,1)$ is infeasible.
\end{proposition}
\begin{proof}
	Among the 24 vertices in $U$, denote by $A_1,A_2$ the two vertices of local profile $(0,7,1,0)$, by $C$ the vertex of local profile $(1,6,0,1)$, and by $D$ the vertex of local profile $(4,0,3,1)$. Since $n_0=5$ in this case, there is a unique type-0 $K_5$, denoted by $Q_4$, not containing $D$. The residual data above shows that $C$ and $D$ always share a type-3 $K_5$, and this is valid for both orbits. Therefore, $C\in Q_4$. Let $Q_4=\{C,B_1,B_2,B_3,B_4\}$, where $B_i$ has local profile $(1,5,2,0)$, $i=1,2,3,4$. The remaining 16 vertices with local profile $(1,5,2,0)$, denoted by $B_5,\cdots, B_{20}$, must each be contained in a type-0 $K_5$ commonly with $D$. Therefore, the vertices in $U$ that share a type-2 $K_5$ with $D$ are precisely $A_1,A_2,B_1,B_2,B_3,B_4$. Since $n_2(D)=3$, by pigeonhole principle, there exist $1\le i<j\le 4$ such that $B_i$ and $B_j$ share the same type-2 $K_5$ with $D$, but $B_i$ and $B_j$ already share a common type-0 $K_5$, namely $Q_4$, which is a contradiction!
\end{proof}

\subsection{The $n_3=0$ cases}

The two $n_3=0$ cases are easier because there are no type-$3$ blocks. After local-signature enumeration, orbit reduction, arithmetic stage, and the type-$0$ geometry filter, the remaining branches are not sent to SAT at all. Instead, we apply two exact combinatorial tests:
\begin{enumerate}[leftmargin=1.8em]
\item a block-compatibility test, where we check whether the assignment of 6 type-0 $K_5$'s can be compatible with the incidences to type-2 $K_5$'s for each outside vertex;
\item a type-2 edge-disjointness audit that checks whether the assigned type-2 $K_5$'s are actually edge-disjoint.
\end{enumerate}
The numerical data are given in Table~\ref{tab:n3zero}.

\begin{table}[]
\centering
\setlength{\tabcolsep}{3pt}
\small
\caption{Enumeration outcomes for the $n_3=0$ prism and grid cases.}
\label{tab:n3zero}
\begin{tabular}{lrrrrrr}
\toprule
pattern & local sig. & orbit reps. & global dist. & type-0 surv. & block-compat. surv. & final surv. \\
\midrule
grid  & 369 & 14 & 79 & 38 & 35 & 0 \\
prism & 377 & 54 & 79 & 38 & 35 & 0 \\
\bottomrule
\end{tabular}
\end{table}

In both cases, exactly 3 branches are eliminated by block compatibility, and all remaining $35$ branches fail the final type-$2$ edge-disjointness audit.

\subsection{The $n_3=4$ cases}

The two cases with $n_3=4$ are handled more efficiently by a \emph{template} method. A template is a multiplicity distribution of local signatures, finer than the global profile distributions used in the standard pipeline. Since only two type-$0$ quintuples are present, there are just two geometries to consider: the two type-$0$ blocks are either disjoint or they share one outside vertex. For each geometry one enumerates all possible templates. Each template is then assigned to the $24$ outside vertices; the forced type-$0$, type-$2$, and type-$3$ quintuples are reconstructed; repeated edges are rejected immediately; and any surviving residue is tested for type-$1$ completion.

We record the data in Table \ref{tab:n3-4-counts}. For the prism pattern, the local-signature enumeration yields $33$ signatures. The disjoint type-$0$ subcase yields $45$ templates and $1494$ branches; $1458$ fail by repeated edges and the remaining $36$ because some residual edge has no type-$1$ completion. The share-$1$ subcase yields $110$ templates and $1424$ branches; $1422$ fail by repeated edges and the remaining two by the residual type-$1$ test.

For the grid pattern there are two orbits. In orbit $\mathcal O_1$, the local-signature enumeration yields $31$ signatures; the disjoint subcase yields $15$ templates and $480$ branches, of which $470$ fail the audit and the remaining $10$ the type-$1$ completion; the share-$1$ subcase yields $80$ templates and $1040$ branches, all excluded immediately by repeated edges. In orbit $\mathcal O_2$, the local-signature enumeration again yields $31$ signatures; the disjoint subcase yields $45$ templates and $1494$ branches, of which $1458$ fail by repeated edges and the remaining $36$ by type-$1$ infeasibility; the share-$1$ subcase yields $108$ templates and $1404$ branches, all excluded by repeated edges.

Thus both $n_3=4$ cases are impossible.

\begin{table}[]
	\centering
	\small
	\caption{Enumeration outcomes for the $n_3=4$ prism and grid cases.}
	\label{tab:n3-4-counts}
	\resizebox{\textwidth}{!}{%
		\begin{tabular}{lllrrrr}
			\toprule
			Pattern & Orbit & Type-0 subcase & Templates & Branches & Repeated-edge rejections & No type-1 candidate \\
			\midrule
			Prism & unique & disjoint & $45$  & $1494$ & $1458$ & $36$ \\
			Prism & unique & share-1  & $110$ & $1424$ & $1422$ & $2$ \\
			\midrule
			Grid & $\mathcal O_1$ & disjoint & $15$  & $480$  & $470$  & $10$ \\
			Grid & $\mathcal O_1$ & share-1  & $80$  & $1040$ & $1040$ & $0$ \\
			Grid & $\mathcal O_2$ & disjoint & $45$  & $1494$ & $1458$ & $36$ \\
			Grid & $\mathcal O_2$ & share-1  & $108$ & $1404$ & $1404$ & $0$ \\
			\bottomrule
		\end{tabular}%
	}
\end{table}

\subsection{Conclusion of the algorithmic exclusion}

The ten cases are now all excluded, so we obtain the main algorithmic result.

\begin{theorem}
There is no decomposition of $K_{33}$ into $51$ copies of $K_5$ and $6$ copies of $K_3$.
\end{theorem}

The full data generated during the algorithmic exclusion is deposited in \cite{k33-decomp}. 

Since the only other arithmetic possibility for $57$ blocks had already been ruled out by the general theory, Theorem~\ref{thm:main-57} follows immediately. Thus we have:
\begin{corollary}
$\Cxi(33,\{3,4,5\},2)\ge 58$.
\end{corollary}

\section{One more step: excluding $58$ blocks}

We now present the counting argument that upgrades the bound from $58$ to $59$.

\begin{proof}[Proof of Theorem~\ref{thm:main-59}]
Assume that $K_{33}$ admits a decomposition into $58$ cliques of order 3, 4 or 5. Solving
\[
\alpha+\beta+\gamma=58,
\qquad
3\alpha+6\beta+10\gamma=\binom{33}{2}=528,
\]
we get only two nonnegative integer solutions:
\[
(\alpha,\beta,\gamma)=(0,13,45)
\qquad\text{or}\qquad
(\alpha,\beta,\gamma)=(4,6,48).
\]

\paragraph{Case $(\alpha,\beta,\gamma)=(0,13,45)$.}
At a vertex $x$,
\[
3\beta_x+4\gamma_x=32,
\]
so $\beta_x\equiv 0\pmod 4$ and hence $\beta_x\in\{0,4,8\}$. If some vertex had $\beta_x=8$, then the $24$ vertices sharing a $K_4$ with it would each satisfy $\beta_y\ge 4$, which would force
\[
\beta\ge \frac{8+4\cdot 24}{4}=26,
\]
a contradiction. Thus $\beta_x\in\{0,4\}$ for all vertices, so exactly $13$ vertices satisfy $\beta_x=4$ and the remaining $20$ satisfy $\beta_x=0$.

The $13$ copies of $K_4$ cover $13\cdot 6=78$ edges, which is exactly $e(K_{13})$. Hence those $13$ vertices form a $K_{13}$ completely decomposed into 13 $K_4$'s, and no edge inside this distinguished $K_{13}$ is covered by a $K_5$. But each of those $13$ vertices satisfies $\gamma_x=5$, so counting incidences with the $K_5$'s gives
\[
\gamma\ge 5\cdot 13=65,
\]
contrary to $\gamma=45$.

\paragraph{Case $(\alpha,\beta,\gamma)=(4,6,48)$.}
At a vertex $x$,
\[
2\alpha_x+3\beta_x+4\gamma_x=32,
\]
so $2\alpha_x+3\beta_x\equiv 0\pmod 4$, and in particular $\beta_x$ is even. If some vertex had $\beta_x\ge 4$, then the vertices sharing a $K_4$ with it would force too many $K_4$ incidences; hence $\beta_x\in\{0,2\}$ for all vertices. The possible nonzero local pairs are therefore
\[
(\alpha_x,\beta_x)=(1,2),\ (2,0),\ (3,2),\ (4,0).
\]
Let $\lambda_1,\lambda_2,\lambda_3,\lambda_4$ be the numbers of vertices of these four types. Counting triangle and $K_4$ incidences gives
\[
\lambda_1+2\lambda_2+3\lambda_3+4\lambda_4=3\alpha=12,
\qquad
2\lambda_1+2\lambda_3=4\beta=24.
\]
Thus $\lambda_1=12$ and $\lambda_2=\lambda_3=\lambda_4=0$. So there are exactly $12$ vertices with $(\alpha_x,\beta_x)=(1,2)$, and the remaining $21$ vertices lie only in $K_5$'s.

Let $U$ be the set of these $12$ active vertices. Among them, the $4$ triangles and $6$ copies of $K_4$ cover
\[
4\cdot 3 + 6\cdot 6 = 48
\]
incidences of edges, so the number of uncovered edges inside $U$ is
\[
\binom{12}{2} - 4\cdot 3 - 6\cdot 6 = 18.
\]
These $18$ edges must be covered by the $48$ quintuples. Let $n_s$ be the number of quintuples meeting $U$ in exactly $s$ vertices. Then
\[
\sum_{s=0}^4 n_s = 48,
\qquad
\sum_{s=0}^4 s n_s = 72,
\qquad
\sum_{s=0}^4 \binom{s}{2}n_s = 18.
\]
But for each $s\in\{0,1,2,3,4\}$ we have $\binom{s}{2}\ge s-1$, so
\[
18 = \sum_{s=0}^4 \binom{s}{2}n_s
\ge \sum_{s=0}^4 (s-1)n_s
= \sum_{s=0}^4 s n_s - \sum_{s=0}^4 n_s
= 72-48=24,
\]
a contradiction. In other words, the total $K_5$'s count 48 would force a larger amount of edges inside $U$ to be covered by quintuples. 

Both arithmetic possibilities are impossible, so no $58$-block decomposition exists.
\end{proof}

\section{Connection with the packing number $D(33,5,2)$}

The mixed-covering problem for $K_{33}$ we studied in this paper is closely related to the unresolved packing number $D(33,5,2)$, known to lie between $48$ and $51$ \cite{BrouwerCWC,Johnson1972}. If a packing of $51$ copies of $K_5$ in $K_{33}$ existed, then its leave would be a $4$-regular graph on $9$ vertices. Up to isomorphism, there are exactly $16$ simple $4$-regular graphs on $9$ vertices; since any such graph is necessarily connected, this is exactly the number of connected quartic graphs of order $9$, see \cite{MeringerRegularGraphs}. The present $57$-block exclusion shows that two natural leave patterns, namely the two obtained by decomposing the leave into six triangles on nine vertices (the grid and prism patterns), cannot occur. In this sense, the present paper is not only about a covering number: it also removes two candidate leaves from the $51$-packing problem.

\section{Conclusion}

We have shown that an exact mixed-clique covering of $K_{33}$ by only $57$ blocks of order 3, 4 or 5 does not exist, and a further counting argument excludes the $58$-block case as well; therefore, the lower bound for $C^\xi(33,\{3,4,5\},2)$ is upgraded to $59$.

From a methodological perspective, the main point is that the proof is not a monolithic SAT or ILP certificate. It is a progressive exact search built from several layers of reduction: symmetry, local signature enumeration, arithmetic distributions, partial geometry filter, SAT feasibility on reduced branches, and post-SAT combinatorial audits. This framework appears flexible enough to attack other hard covering and packing problems in which direct exact formulations are too large but structured residues remain accessible.

\paragraph*{Acknowledgement} This work was partially supported by Grant of SGS No. SP2026/019 and No. SP2026/050, VSB -- Technical University of Ostrava, Czech Republic. Petr Kov\'a\v{r} and Yifan Zhang are also co-funded by the financial support of the European Union under the REFRESH – Research Excellence For Region Sustainability and High-tech Industries project number CZ.10.03.01/00/22 003/0000048 via the Operational Programme Just Transition. 
\printbibliography

\end{document}